\documentclass{amsart}
\usepackage{color}
\usepackage{graphicx,epsf}
\usepackage{amsmath,amscd}
\newtheorem{theorem}{Theorem}[section]
\newtheorem{lemma}[theorem]{Lemma}
\newtheorem{corollary}[theorem]{Corollary}
\newtheorem{proposition}[theorem]{Proposition}
\theoremstyle{definition}
\newtheorem{definition}[theorem]{Definition}
\newtheorem{example}[theorem]{Example}

\theoremstyle{remark}
\newtheorem{remark}[theorem]{Remark}
\theoremstyle{conjecture}
\newtheorem{conjecture}[theorem]{Conjecture}
\numberwithin{equation}{section}

\newcommand{\Real}{{\mathbb R}}

\newcommand{\R}{{\rm  R}}

\newcommand{\eps}{\varepsilon}

\newcommand{\x}{\mathbf{x}}
\newcommand{\y}{\mathbf{y}}
\newcommand{\z}{\mathbf{z}}

\newcommand {\hide}[1]{}

\begin{document}

\title[Semi-monotone sets]{Semi-monotone sets}
\author{Saugata Basu}
\address{Department of Mathematics,
Purdue University, West Lafayette, IN 47907, USA}
\email{sbasu@math.purdue.edu}
\author{Andrei Gabrielov}
\address{Department of Mathematics,
Purdue University, West Lafayette, IN 47907, USA}
\email{agabriel@math.purdue.edu}
\author{Nicolai Vorobjov}
\address{
Department of Computer Science, University of Bath, Bath
BA2 7AY, England, UK}
\email{nnv@cs.bath.ac.uk}
\thanks{S.~Basu was supported in part by NSF grant CCF-0915954,
A.~Gabrielov was supported in part by NSF grant DMS-0801050.}


\begin{abstract}
A coordinate cone in $\Real^n$ is an intersection of some coordinate
hyperplanes and open coordinate half-spaces.
A semi-monotone set is an open bounded subset of $\Real^n$, definable in an o-minimal structure
over the reals, such that its intersection with any translation of any
coordinate cone is connected.
This can be viewed as a generalization of the convexity property.
Semi-monotone sets have a number of interesting geometric and combinatorial properties.
The main result of the paper is that every semi-monotone set is a topological regular cell.
\end{abstract}
\maketitle

\section*{Introduction}
It is well known that in o-minimal geometry, definable sets that are
locally closed are easier to handle than arbitrary definable sets.
A typical example of this phenomenon can be seen in the well-studied
problem of obtaining tight upper bounds on  topological invariants such as the
Betti numbers of semi-algebraic or semi-Pfaffian sets in terms of the
complexity of formulae defining them.
Certain standard techniques from  algebraic topology (for example, inequalities
stemming from the Mayer-Vietoris exact sequence) are
directly applicable only in the case of locally closed definable sets.
Definable sets which are not locally closed are comparatively more difficult to analyze.
In order to overcome this difficulty, Gabrielov and Vorobjov
in their paper \cite{GV09} suggested a construction which, given a definable set $S$ in an
o-minimal extension of the reals, produced an explicit family of
definable compact sets converging to $S$.
Under a certain technical condition (called ``separability'')
they proved that the  approximating compact sets are homotopy equivalent to $S$.
The separability condition is automatically
satisfied in many cases of interests -- such as when $S$ is described by equations
and inequalities with continuous definable functions.

However, the property of separability is not  preserved under taking images of definable maps,
and this restricts the applicability of this construction.
It was conjectured in \cite{GV09} that the crucial property of the approximating
family (homotopy equivalence to $S$) remains true
even without the separability hypothesis.
Proving this conjecture seems to be a rather difficult problem at present.
One of the authors of the current paper (Gabrielov) has outlined a research program whose
completion would lead (amongst other things) to a proof of the conjecture.
The goal of the program is a ``triangulation'' of an increasing definable family of
compact sets.
More precisely, the
goal is to prove that given any increasing definable family of compact sets converging to a
definable set $S \subset \Real^n$, there exists a definable triangulation of $\Real^n$
such that inside each open simplex of this triangulation the increasing definable
family belongs to a finite list of combinatorial types.
Such a triangulation should be considered as being compatible with the given increasing family
(thus generalizing the standard notion of definable triangulations compatible
with a given definable set).
The homotopy equivalence conjecture will then follow from this triangulation.

One of the key steps in Gabrielov's program
is to prove the existence of a regular triangulation of the graph of a
definable function.
More precisely, there is the following conjecture.

\begin{conjecture}\label{con:triang}
Let $f:\> K \to \Real$, be a definable function on a compact definable set $K \subset \Real^m$.
Then there exists a definable triangulation of $K$ such that, for each $n\le\dim K$ and for each
open $n$-simplex $\Delta$ of the triangulation,
\begin{enumerate}
\item
the graph $\Gamma :=\{(\x,t)|\> \x\in\Delta,\,t=f(\x)\}$ of the restriction of
$f$ on $\Delta$ is a regular $n$-cell (see Definition~\ref{def:cell});
\item
either $f$ is a constant on $\Delta$ or
each non-empty level set $\Gamma \cap\{t= {\rm const} \}$ is a regular $(n-1)$-cell.
\end{enumerate}
\end{conjecture}

It should be pointed out that Conjecture~\ref{con:triang} does not follow from results
in the literature
on the existence of definable triangulations adapted to a given finite
family of definable subsets of $\Real^n$ (such as \cite{VDD, Coste}), since all the proofs
use a preparatory linear change of coordinates in order for the given definable sets
to be in a good position with respect to coordinate projections.
Since we are concerned with the graphs and the level sets of a function, in order
to prove Conjecture~\ref{con:triang} we are not allowed to make any change of coordinates
which involves the last coordinate.
Thus, the standard methods of obtaining a definable triangulation using
``cylindrical decomposition'' are not immediately applicable.
In the book \cite{VDD}, van den Dries describes
a strong form of cylindrical decomposition in which the cells are defined by
functions having coordinate-wise monotonicity property (such cells are
called regular in \cite{VDD}).
We show that in fact these cells are not necessarily
regular cells in the sense of topology (see Definition~\ref{def:cell}).
To prove Conjecture~\ref{con:triang}, we need a sufficiently general class of definable sets
which are guaranteed to be topologically regular cells.

In this paper, we introduce a new class of definable sets,
which we call semi-monotone sets, and show
that an open definable semi-monotone set in $\Real^n$ is a regular $n$-cell.
A coordinate cone in $\Real^n$ is an intersection of some coordinate
hyperplanes and open coordinate half-spaces.
A semi-monotone set is a definable in an o-minimal structure over the reals,
open bounded subset of $\Real^n$ such that its intersection with any translation of any
coordinate cone is connected.
It is obvious that every convex definable bounded open set is semi-monotone.
Some non-convex examples
as well as counter-examples  are shown in
Figure~\ref{fig:examplesandcounterexamples}.
The paper is organized as follows.
In Section~\ref{sec:definitions} we define a semi-monotone set and prove necessary
and sufficient conditions for an open bounded set to be semi-monotone, which are similar
to the properties of cylindrical cells in o-minimal geometry.
In particular, it is proved that any semi-monotone set is a ``band'' between the graphs
of two semi-continuous functions which are defined on a semi-monotone set of a smaller
dimension and satisfy certain monotonicity properties.

Section~\ref{sec:regular} contains the proof of the main result, that every semi-monotone set
is a regular cell.
In Section~\ref{sec:real-closed} we prove the regularity in the case of semi-algebraic
semi-monotone sets defined over an arbitrary real closed field.
In Section~\ref{sec:VDD} we show that cylindrical cells called ``regular'' in \cite{VDD} are
not necessarily topologically regular.

In Section~\ref{sec:boolean} a concept of a regular Boolean function is introduced.
A Boolean function $\psi(\xi_1,\dots,\xi_n)$ in $n$ Boolean variables
$\xi_j\in\{0,1\}$ is called {\em regular} if the result of any sequence of operations
$\forall \xi_j$ and $\exists \xi_k$ applied to $\psi$ does not depend on the order of the operations.
To every point $p$ outside a given open bounded set $U$ we assign a Boolean function,
taking the value 1 exactly on the octants with the vertex $p$ which have a non-empty intersection
with $U$ (see Definition~\ref{def:assigne-boolean}).
The main result of Section~\ref{sec:boolean} is that $U$ is semi-monotone if and only if
the functions assigned to all points outside $U$ are regular.

Section~\ref{sec:appendix} is the Appendix containing some known and new facts from PL topology
needed in the proof of the main result.

\subsection*{Acknowledgements} We thank S.~Ferry, J.~McClure, and N.~Mnev for useful discussions.

\section{Equivalent definitions of semi-monotone sets}\label{sec:definitions}

In what follows we fix an o-minimal structure over $\Real$, and consider only sets and maps
that are definable in this structure.

\begin{definition}
\label{new}
Let
$$X_{j,\sigma,c}:=\{\x=(x_1, \ldots , x_n) \in \Real^n|\>  x_j\,\sigma\,c\},$$
for $1\le j\le n,\; \sigma\in\{<,=,>\}$ and $c\in\Real$.
An open (possibly, empty) bounded set $U\subset\Real^n$ is called {\sl semi-monotone}
if
$$U \cap X_{j_1,\sigma_1,c_1}\cap \cdots \cap X_{j_k,\sigma_k,c_k}$$
is connected for any $0\le k \le n$, any $1\le j_1< \dots <j_k \le n$, any
$\sigma_1,\dots,\sigma_k$ in $\{<, =, >\}$, and any $c_1, \dots, c_k \in \Real$.
\end{definition}

\vspace*{-1in}
\begin{figure}[hbt]
         \centerline{
           \scalebox{0.60}{
             \includegraphics{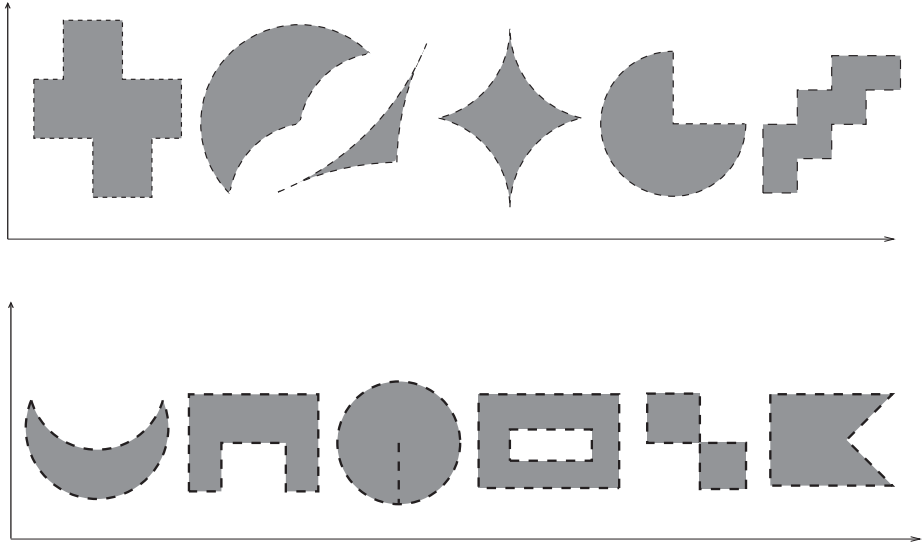}
             }
           }
\vspace*{-1.5in}
\caption{Top: examples of semi-monotone subsets of the plane.
Bottom: examples of open subsets of the plane
which are not semi-monotone.}
         \label{fig:examplesandcounterexamples}
 \end{figure}

\begin{lemma}\label{le:projection}
The projection of a semi-monotone set $U$ on any coordinate subspace is a semi-monotone set.
\end{lemma}

\begin{proof}
Let $U'$ be the projection of $U$ on the subspace of coordinates $x_1, \ldots ,x_m$ where
$m \le n$.
Then any intersection
$$U' \cap X'_{j_1,\sigma_1,c_1} \cap \cdots \cap X'_{j_k,\sigma_k,c_k},$$
where $X'_{j,\sigma,c}=\{ (x_1, \ldots ,x_m) \in \Real^m|\> x_j\,\sigma\,c \}$ and
$1 \le j_1< \cdots <j_k \le m$, is connected as the projection of a connected set
$$U \cap X_{j_1,\sigma_1,c_1} \cap \cdots \cap X_{j_k,\sigma_k,c_k}.$$
\end{proof}

\begin{theorem}\label{th:inductive}
An open and bounded set $U\subset\Real^n$ is semi-monotone
if and only if both of the following conditions hold:
\begin{itemize}
\item[(S1)]
its intersection with any line parallel to the $x_n$-axis
is either empty or an open interval,
\item[(S2)]
projections of the sets $U\cap X_{n,\sigma,c}$ to
$\Real^{n-1}$ along the $x_n$-axis are semi-monotone sets
in $\Real^{n-1}$ for any $\sigma\in\{<,=,>\}$ and $c\in\Real$.
\end{itemize}
\end{theorem}

\begin{proof}
We prove the statement by induction on $n$. For $n=1$ it is obvious.
Let $U$ satisfy (S1) and (S2).
The set $U$ is connected, otherwise its projection $U'$ along
the $x_n$-axis would be not connected (a very special case
of the Vietoris-Begle theorem, \cite{Spanier}).
This would contradict (S2), since $U=U\cap X_{n,<,c}$ for large positive $c$ and hence
its projection is connected.

For $j_k<n$, the projection of
$$U\cap X_{j_1,\sigma_1,c_1}\cap\dots\cap X_{j_k,\sigma_k,c_k}$$
to $\Real^{n-1}$ is equal to
$$U'\cap X'_{j_1,\sigma_1,c_1}\cap\dots\cap X'_{j_k,\sigma_k,c_k},$$
where $X'_{j,\sigma,c}=\{ \x=(x_1, \ldots ,x_{n-1}) \in \Real^{n-1}|\> x_j\,\sigma\,c \}$.
This set is connected by the inductive hypothesis, hence
$$U\cap X_{j_1,\sigma_1,c_1}\cap\dots\cap X_{j_k,\sigma_k,c_k}$$
is connected, by the Vietoris-Begle theorem.

For $j_k=n$, the projection of
$$U\cap X_{j_1,\sigma_1,c_1}\cap\dots\cap X_{j_k,\sigma_k,c_k}$$
to $\Real^{n-1}$ is equal to the intersection of the projection of
$U\cap X_{n,\sigma_k,c_k}$ and the set
$$X'_{j_1,\sigma_1,c_1}\cap\dots\cap X'_{j_{k-1},\sigma_{k-1},c_{k-1}}.$$
It is connected due to condition (S2) and the induction hypothesis,
hence
$$U\cap X_{j_1,\sigma_1,c_1}\cap\dots\cap X_{j_k,\sigma_k,c_k}$$
is connected, again by the Vietoris-Begle theorem.

Conversely, if $U$ is a semi-monotone set,
its intersection with each line parallel to any coordinate axis
is connected, i.e., either empty or an open interval.
Since all sets
$$U\cap X_{n,\sigma,c}\cap X_{j_1, \sigma_1,c_1}\cap \dots \cap X_{j_k, \sigma_k, c_k}$$
are connected, their projections along the $x_n$-axis are connected.
This implies that projections of the sets
$U\cap X_{n,\sigma,c}$ along the $x_n$-axis are semi-monotone sets in $\Real^{n-1}$.
\end{proof}

\begin{corollary}\label{slice}
\begin{enumerate}
\item
If $U \subset \Real^n$ is a semi-monotone set, then $U\cap X_{j, <, a}
\cap X_{j, >, b}$
is a semi-monotone set in $\Real^n$ for any $1\le j\le n$ and $a, b \in\Real$.
\item
If $U \subset \Real^n$ is a semi-monotone set, then
$$U\cap X_{j_1,\sigma_1,c_1}\cap \cdots \cap X_{j_k,\sigma_k,c_k}$$
is semi-monotone for any $0\le k \le n$, any $1\le j_1< \dots <j_k \le n$,
any $\sigma_1,\dots,\sigma_k$ in $\{<, =, >\}$, and any $c_1, \dots, c_k \in \Real$.
\end{enumerate}
\end{corollary}

\begin{proof}
(1)\ The statement is obvious for $n=1$.
Let $n \ge 2$ and $j<n$.
Then the set $U\cap X_{j, <, a} \cap X_{j, >, b}$ satisfies conditions (S1) and (S2),
hence it is semi-monotone.

(2)\ Immediately follows from (1).
\end{proof}

\begin{corollary}\label{co:acyclic}
Any semi-monotone set $U$ is acyclic.
\end{corollary}

\begin{proof}
We prove the statement by induction on $n$.
The base, for $n=1$ is obvious.
Applying Theorem~\ref{th:inductive}, (S2), to $U=U\cap X_{n,<,c}$ for a large positive $c$
we conclude that the projection $U'$ of $U$ along $x_n$-axis is a semi-monotone set,
and, therefore, by the inductive hypothesis, is acyclic.
By (S1), the fibres of this projection map are acyclic, so, since the projection
is an open map, the Vietoris-Begle theorem implies that $U$ is also acyclic.
\end{proof}

Note that in Theorem~\ref{regularcell} we will prove a much stronger result.

\begin{definition}
A bounded upper semi-continuous function $f$ defined on a semi-monotone
set $U \subset\Real^n$ is {\em submonotone} if, for any $r > \inf_{\x\in U} f(x)$, the set
$$\{\x \in U|\> f(\x)< r\}$$
is semi-monotone.
A function $f$ is {\em supermonotone} if $(-f)$ is submonotone.
\end{definition}

\begin{theorem}\label{th:band}
An open and bounded set $U \subset \Real^n$ is semi-monotone if and only if it satisfies
the following conditions.
If $U \subset \Real^1$ then $U$ is an open interval.
If $U \subset \Real^n$, then
$$U=\{(\x,t)|\> \x \in U',\; f(\x)<t<g(\x)\}$$
for some functions $f$ and $g$ on a semi-monotone
set $U' \subset\Real^{n-1}$, where $f(\x)<g(\x)$ for all $\x \in U'$, with
$f(\x)$ being submonotone and $g(\x)$ being supermonotone.
\end{theorem}

\begin{proof}
Suppose that $U$ is semi-monotone, and $U'$ is the projection of $U$ on the subspace of
coordinates $x_1, \ldots ,x_{n-1}$.
By the Lemma~\ref{le:projection}, $U'$ is a semi-monotone set.
According to (S1) of the Theorem~\ref{th:inductive}, any fibre of the projection map over
a point $\x \in U'$ is an open interval.
Take the lower endpoints of these intervals as values of $f$ and upper endpoints as values of $g$.
It follows that
$$U=\{(\x,t)|\> \x\in U',\;f(\x)<t<g(\x)\}.$$

The function $f$ is bounded because $U$ is bounded.
The function $f$ is upper semi-continuous since otherwise there would exist a sequence
$\x^{(i)} \in U'$ with $\lim_{i \to \infty} \x^{(i)}=\x^{(0)} \in U'$ such that
$\lim_{i \to \infty} f(\x^{(i)})- f(\x^{(0)}) > \eps$ for some positive $\eps \in \Real$.
Then the interval with the lower endpoint $f(\x^{(0)})$ has a point belonging both to $U$
and to the boundary of $U$, which contradicts to the openness of $U$.

Let $r > \inf_{\x \in U'} f(\x)$.
The definition of $f$ implies that the set
$$S_r :=\{\x \in U'|\> f(\x)< r\}$$
is the projection on the subspace of coordinates
$x_1, \ldots ,x_{n-1}$ of the intersection $U \cap X_{n,<,r}$.
According to the Corollary~\ref{slice}, $U \cap X_{n,<,r}$ is a semi-monotone set, thus,
by the Lemma~\ref{le:projection}, its projection $S_r$ is also semi-monotone.
It follows that $f$ is submonotone.

Similarly, the function $g$ is supermonotone.

We proved that the semi-monotone $U$ satisfies the conditions in the theorem.
Now assume that an open and bounded set $U \subset \Real^n$ satisfies these conditions,
in particular, its projection $U'$ is semi-monotone.
We prove that $U$ is semi-monotone by induction on $n$, the base for $n=1$ being trivial.

According to the Theorem~\ref{th:inductive}, it is sufficient to prove that $U$ satisfies
conditions (S1) and (S2).
The condition (S1) holds true because every intersection of $U$ with a straight line
parallel to $x_n$-axis is an interval $(f(\x), g(\x))$ for some $\x \in U'$.

For any $c \in (\inf_{\x \in U'}f(\x), \sup_{\x \in U'}g(\x))$
the projection of the set $U \cap X_{n,< ,c}$ to the subspace of coordinates $x_1, \ldots, x_{n-1}$
coincides with $\{\x \in U'|\> f(\x)< c \}$ and therefore is a semi-monotone set.
Similarly, the projection of $U \cap X_{n,> ,c}$ is a semi-monotone set.
By Vietoris-Begle theorem, both sets $U \cap X_{n,< ,c}$ and $U \cap X_{n,> ,c}$ are connected.

To satisfy condition (S2) of the Theorem~\ref{th:inductive}
it remains to prove that the projection $W$ of $U \cap X_{n, = ,c}$ is also a semi-monotone set.
We will prove this by showing that any intersection of the kind
$W \cap X_{j_1,\sigma_1,c_1}\cap \cdots \cap X_{j_k,\sigma_k,c_k}$ is connected, where
$j_1 < \cdots <j_k < n$ and $\sigma_1, \ldots ,\sigma_k \in \{ <, =, > \}$.
For this, it is enough to prove that any intersection of the kind
$U \cap X_{n, = ,c} \cap X_{j_1,\sigma_1,c_1}\cap \cdots \cap X_{j_k,\sigma_k,c_k}$ is connected.
If at least one $\sigma_i$ is $=$, then the connectedness
follows from the inductive hypothesis, since the conditions in the theorem are
compatible with the translated coordinate cones
$X_{j_1,\sigma_1,c_1}\cap \cdots \cap X_{j_k,\sigma_k,c_k}$.
Otherwise, $U \cap X_{j_1,\sigma_1,c_1}\cap \cdots \cap X_{j_k,\sigma_k,c_k}$ is itself a bounded
open set in $\Real^n$ satisfying the conditions of the theorem, and it remains to prove that
the intersection of this set with $X_{n, = ,c}$ or, without a loss of generality, the
intersection $U \cap X_{n, = ,c}$, is connected.

Suppose that $U \cap X_{n, = ,c}$ is not connected.
Every fibre $U \cap X_{n, = ,c} \cap X_{n-1, =,t}$ of the projection of $U \cap X_{n, = ,c}$
on the $x_{n-1}$-axis is a semi-monotone cell by the inductive hypothesis hence, by the
Corollary~\ref{co:acyclic}, is acyclic.
Then Vietoris-Begle theorem implies that the image of the projection of $U \cap X_{n, = ,c}$
also is not connected, i.e., there is a point $t_0$ such that
$U \cap X_{n, = ,c} \cap X_{n-1, =,t_0}= \emptyset$ while both sets
$U \cap X_{n, = ,c} \cap X_{n-1, <,t_0}$ and $U \cap X_{n, = ,c} \cap X_{n-1, >,t_0}$ are
non-empty.
Because $U$ is open while the sets $U \cap X_{n, < ,c}$ and $U \cap X_{n, > ,c}$ are
connected, each of them has a non-empty intersection with $X_{n-1, =,t_0}$.
But this implies that $U \cap X_{n-1, =,t_0}$ is not connected which contradicts what
was proved before.
\end{proof}

\section{Semi-monotone sets are regular cells}\label{sec:regular}

Any compact definable set in $\Real^n$ admits a finite triangulation
(see, e.g., \cite{VDD}), in particular is definably homeomorphic to a polyhedron.
Any open set in $\Real^n$ is a polyhedron.

\begin{definition}\label{def:cell}
A definable set $U$ is called a {\em regular $k$-cell} if the pair $(\overline U, U)$ is
definably homeomorphic to the pair $([-1,1]^k, (-1,1)^k)$.
\end{definition}

In this section we say that a definable set {\em is a closed $n$-ball} if it is definably
homeomorphic to $[-1,1]^n$, {\em is an open $n$-ball} if it is definably
homeomorphic to $(-1,1)^n$, and {\em is an $(n-1)$-sphere} if it is definably
homeomorphic to $[-1,1]^n \setminus (-1,1)^n$.

Proposition~\ref{pr:le1.10} implies that if $U \subset \Real^n$ is an open definable set, then
$U$ is a regular cell if and only if $\overline U$ is an $n$-ball and
the {\em frontier} $\overline U \setminus U$ is an $(n-1)$-sphere.

\begin{theorem}\label{regularcell}
A semi-monotone set $U \subset \Real^n$ is a regular $n$-cell.
\end{theorem}

We are going to prove Theorem \ref{regularcell}
by induction on the dimension $n$ of a regular cell.
For $n=1$ the statement is obvious.
Assume it to be true for $n-1$.

\begin{lemma}\label{cut}
Let $U \subset \Real^n$ be a semi-monotone set.
Let
$$U_0:=U\cap X_{j,=,c},\quad U_+ :=U\cap X_{j,>,c},\quad \text{and}\quad U_- :=U\cap X_{j,<,c}$$
for some $1 \le j \le n$ and $c \in \Real$.
Then $\overline U_+\cap \overline U_-=\overline U_0$.
\end{lemma}

\begin{proof}
Let a point $\x=(x_1, \ldots ,x_n) \in X_{j,=,c} \setminus \overline U_0$
belong to $\overline U_+\cap \overline U_-$.
Then there is an $\eps >0$ such that an open cube centered at $\x$,
$$C_\eps := \bigcap_{1 \le j \le n} \{(y_1, \ldots ,y_n)|\> |x_j - y_j| < \eps \}
\subset \Real^n,$$
has non-empty intersections with both $U_+$ and $U_-$
and the empty intersection with $U_0$.
Thus, $C_\eps \cap U$ is not connected, which is not possible since, according to
Corollary~\ref{slice} (1), $C_\eps \cap U$ is semi-monotone.
\end{proof}

\begin{corollary}\label{union}
Let $U \subset \Real^n$ be a semi-monotone set.
If $U_+$ and $U_-$ in Lemma \ref{cut} are regular cells, then $U$ is a regular cell.
\end{corollary}

\begin{proof}
We need to prove that $\overline U$ is a closed $n$-ball, and that the frontier
$\overline U \setminus U$ is an $(n-1)$-sphere.
The only non-trivial case is when $U_0$ is non-empty.

Since $U_0$ is semi-monotone due to Corollary~\ref{slice}, $U_0$ is a regular $(n-1)$-cell
by the inductive hypothesis.
It follows that $\overline U_0$, $\overline U_+$, and $\overline U_-$ are closed balls,
while $\overline U_0 \setminus U$ is $(n-2)$-sphere.
Hence $\overline U$ is obtained by gluing together two closed $n$-balls, $\overline U_+$
and $\overline U_-$ along closed $(n-1)$-ball $\overline U_0$ (see Definition~\ref{def:gluing}).
Proposition~\ref{pr:cor3.16} implies that $\overline U$ is a closed $n$-ball.

According to Proposition~\ref{pr:cor3.13}, the sets $\overline U_+ \setminus U
= \partial \overline U_+ \setminus U_0$ and
$\overline U_- \setminus U = \partial \overline U_- \setminus U_0$ are closed $(n-1)$-balls.
The frontier $\overline U \setminus U$ of $U$ is obtained by gluing
$\overline U_+ \setminus U$ and $\overline U_- \setminus U$ along the set
$(\overline U_+\cap \overline U_-) \setminus U$ which, by Lemma~\ref{cut},
is equal to $\overline U_0 \setminus U$ and thus,
is an $(n-2)$-sphere, the common boundary of $\overline U_+ \setminus U$
and $\overline U_- \setminus U$.
It follows from Proposition~\ref{pr:le1.10} that $\overline U \setminus U$ is an $(n-1)$-sphere.
\end{proof}

\begin{lemma}\label{complement}
If $U$ and $U_-$ in Lemma \ref{cut} are regular cells, then $U_+$ is also a regular cell.
\end{lemma}

\begin{proof}
Proposition~\ref{pr:shiota} implies that $\overline U_+$ is
a closed $n$-ball.
By the inductive hypothesis, $U_0$ is a regular cell.
By Proposition~\ref{pr:cor3.13}, $\overline U_+ \setminus U= \partial \overline U_+ \setminus U_0$
is a closed $(n-1)$-ball.
Then the frontier $\overline U_+ \setminus U_+$ of $U_+$ is obtained by gluing two closed
$(n-1)$-balls, $\overline U_+ \setminus U$ and $\overline U_0$ along the $(n-2)$-sphere
$\overline U_0 \setminus U$.
Therefore, by Proposition~\ref{pr:le1.10}, the frontier of $U_+$ is an $(n-1)$-sphere.
\end{proof}

\begin{lemma}\label{both}
Let $n>5$ and $U\subset\Real^n$ be a semi-monotone set and regular cell.
Then, for a generic $c$, both $U_+$ and $U_-$ in Lemma \ref{cut} are regular cells.
\end{lemma}

\begin{proof}
The set $U_0$ is a regular cell by the inductive hypothesis.
Due to the theorem on triangulation of definable functions (\cite{Coste}, Th.~4.5) applied to
the projection on $x_j$-coordinate function, there is a triangulation of $\overline U$ and
a neighbourhood $(a, b)$ of $c$ in $\Real$ such that the polyhedra corresponding to
$\overline U \cap ((a,b) \times \Real^{n-1})$ and $\overline U_0 \times (a,b)$
are PL-homeomorphic.
Hence, the $(n-1)$-sphere $\partial U_0$ is locally flatly embedded in the $n$-sphere $\partial U$.
The lemma now follows from Proposition~\ref{pr:balls}.
\end{proof}

Let $\Real_{+}^{n}$ be the open first octant
$\{ (x_1, \ldots , x_n) \in \Real^n|\> x_j>0\> \text{for all}\> 1 \le j \le n \}$.

\begin{lemma}\label{germ}
Let $U$ be a semi-monotone set in $\Real_{+}^{n}$ such that the origin is in $\overline U$.
Let $c(t)=(c_1(t),\dots,c_n(t))$
be a germ of a generic definable curve inside $U$ converging to the origin as $t \to 0$.
Then
$$U_t :=U \cap \{x_1<c_1(t),\dots,x_n<c_n(t)\}$$
is a regular cell for all small positive $t$.
\end{lemma}

\begin{proof}
Due to the inductive hypothesis of the induction on the dimension $n$, for every $1 \le i \le n$
each $(n-i)$-dimensional semi-monotone set
$$C_{j_1,\dots,j_i,t}:=U \cap\{ x_{j_1}=c_{j_1}(t),\ldots, x_{j_i}=c_{j_i}(t),\;
x_k<c_k(t)\;\text{for all}\; k \neq j_1,\dots,j_i\},$$
where $1 \le j_1 < \cdots <j_i \le n$, is a regular $(n-i)$-cell.
Since $c(t)\in U$, all sets $C_{j_1,\dots,j_i,t}$ are non-empty.

Due to the theorem on triangulation of definable functions (\cite{Coste}, Th.~4.5),
for all small positive $t$, $\overline U_t$ is definably homeomorphic to a closed cone
with the vertex at the origin and the base definably homeomorphic to $\overline D_t$, where $D_t$
is the $(n-1)$-dimensional regular cell complex formed by cells
$C_{j_1,\dots,j_i,t}$ for all $1 \le j_1 < \cdots <j_i \le n$.
Hence it is enough to prove that $D_t$ is shellable (see Definition~\ref{def:shellable}),
and therefore is a regular cell due to Proposition~\ref{pr:shellable}.
We prove by induction on $k=1,\dots,n$ a more general claim that the regular cell complex
$D_{k,t}$ formed by cells $C_{j_1,\dots,j_i,t}$, $1 \le j_1 < \cdots <j_i \le k$ is shellable.

The base case $k=1$ is true because $C_{1,t}$ is a regular $(n-1)$-cell.
By the inductive hypothesis on $n$, the set
$$C_{i,j,t}=U \cap\{x_i=c_i(t),\;x_j=c_j(t),\; x_k<c_k(t)\; \text{for all}\; k\ne i,j\}$$
is a regular $(n-2)$-cell.
Since the germ $c(t)$ is generic,
$\overline C_{i,j,t}= \overline C_{i,t} \cap \overline C_{j,t}$.
Hence $D_{2,t}$ is obtained by gluing together
two regular $(n-1)$-cells, $C_{1,t}$ and $C_{2,t}$,
along a regular $(n-2)$-cell $C_{1,2,t}$
which is their common boundary (see Definition~\ref{def:gluing}).
It follows that the cell complex $D_{2,t}$ is shellable.

By the inductive hypothesis the complex $D_{k,t}$ is shellable.
The set $C_{k+1,t}$ is a regular $(n-1)$-cell whose common boundary with $D_{k,t}$ is
the $(n-2)$-dimensional shellable complex formed by $k$ regular $(n-2)$-cells
$C_{1,k+1,t}, \ldots ,C_{k,k+1,t}$.
By Proposition~\ref{pr:shellable}, this common boundary is a regular $(n-2)$-cell.
Hence, by Proposition~\ref{pr:shellable} again, the complex $D_{k+1,t}$ is shellable.
\end{proof}

\begin{lemma}\label{4D}
Let $U \subset \Real_{+}^{n}$ be a semi-monotone set, with $n \le 5$, such that the origin is
in $\overline U$, and let
$c(t)=(c_1(t),\dots,c_n(t))$ be a germ of a generic definable curve inside $\Real_{+}^{n}$
(not necessarily inside $U$) converging to the origin as $t \to 0$.
Then
$$U_t=U \cap \{x_1<c_1(t),\dots,x_n<c_n(t)\}$$
is a regular cell for all small positive $t$.
\end{lemma}

\begin{proof}
We can repeat the proof of Lemma \ref{germ} if we prove
that the regular cell complex $D_t$ formed by the {\em non-empty} sets
$$C_{j_1,\dots,j_i,t}:=U \cap\{ x_{j_1}=c_{j_1}(t),\dots,x_{j_i}=c_{j_i}(t),\;x_k<c_k(t)\;
\text{for all}\; k\ne j_1,\dots,j_i\},$$
where $1 \le j_1 < \cdots <j_i \le n$,
is shellable.
The difference from the proof of Lemma~\ref{germ} is that here some of the sets
$C_{j_1,\dots,j_i,t}$ may be empty.

By Corollary~\ref{co:acyclic}, the semi-monotone cell $U_t$ is acyclic.
Hence $D_t$ is acyclic, too, for all small $t>0$.

Since $U$ is open, if $C_{j_1,\dots,j_k,t}$ is non-empty
then $C_{i_1,\dots,i_l,t}$ is non-empty for any subset
$\{i_1,\dots,i_l\}$ of $\{j_1,\dots,j_k\}$.
It follows that the complex $D_t$ can be represented as a simplicial subcomplex
$X$ of an $(n-1)$-simplex $\Delta$ so that every non-empty set $C_{j_1,\dots,j_i,t}$ corresponds
to the $(i-1)$-face of $\Delta$ having vertices $j_1, \ldots ,j_i$.

Observe that $X$ is acyclic since $D_t$ is acyclic.
We prove by induction on the number of simplices in $X$ that the acyclicity of $X$ implies
that $D_t$ is shellable.
The base of the induction, for a single vertex is trivial.
According to Proposition~\ref{pr:acyclic}, $X$ has a vertex $v$ with the
acyclic link $L$.
The vertex $v$ corresponds to a regular $(n-1)$-cell $C_{j,t}$, while the link $L$
corresponds to the $(n-2)$-subcomplex of $D_t$ along which $C_{j,t}$ is glued to $D_t$.
By the inductive hypothesis applied to $L$, that subcomplex of $D_t$ is shellable,
and thus, by Proposition~\ref{pr:shellable}, is a regular cell.
Removing the star of $v$ in $X$, we get the subcomplex $Y$ of $X$ which is acyclic by
the Mayer-Vietoris exact sequence.
By the inductive hypothesis, the subcomplex of $D_t$, corresponding to $Y$ is shellable.
We have proved that $D_t$ is obtained by gluing a regular cell to a shellable complex
along a regular cell, hence $D_t$ is shellable.
\end{proof}

\begin{lemma}\label{general_box}
Let $U \subset \Real_{+}^{n}$ be a semi-monotone set, with $n \le 5$,
such that the origin is in $\overline U$, and let $c=(c_1,\dots,c_n) \in \Real_{+}^{n}$.
Then $U_c :=U \cap \{x_1<c_1,\dots,x_n<c_n\}$
is a regular cell for a generic $c$ with a small $\| c \|$ .
\end{lemma}

\begin{proof}
Consider a definable set $U_{\y}:= U \cap \{x_1<y_1,\dots,x_n<y_n\} \subset \Real_{+}^{2n}$
with coordinates $x_1, \ldots ,x_n,y_1, \ldots ,y_n$ and $\y=(y_1, \ldots, y_n)$.
By Corollary~\ref{cor:fibres}, there is a partition of $\Real_{+}^{n}$
(having coordinates $y_1, \ldots, y_n$) into definable sets $T$
such that if any $T$ is fixed, then
for all $\y \in T$ the closures $\overline U_{\y}$ are
definably homeomorphic to the same polyhedron, and the frontiers
$\overline U_{\y} \setminus U_{\y}$ are definably homeomorphic to the same polyhedron.

For every $n$-dimensional $T$, such that the origin is in $\overline T$, there is, by the curve
selection lemma (\cite{Coste}, Th.~3.2) a germ of a generic
definable curve $c(t)$ converging to $0$ as $t \to 0$.
Hence, by Lemma~\ref{4D}, for each $c \in T$ the
set $\overline U_c$ is  a closed $n$-ball, while $\overline U_c \setminus U_c$ is an $(n-1)$-sphere.
Therefore, $U_c$ is a regular cell.
\end{proof}

\begin{lemma}\label{cutting1}
Using the notation from Lemma~\ref{general_box}, for $n \le 5$,
and a generic $c \in \Real_{+}^{n}$ with a small $\| c \|$, the intersection
$$U_c \cap \bigcap_{1 \le \nu \le k}  \{ x_{j_\nu} \sigma_\nu a_\nu \},$$
for any $j_\nu \in \{ 1, \ldots , n \}$, $\sigma_\nu \in \{ <,> \}$, and for any generic
$a_1 \ge \cdots \ge a_k$, is either empty or a regular cell.
\end{lemma}

\begin{proof}
It is sufficient to assume that $a_\nu < c_{j_\nu}$ for all $\nu$.
Induction on $k$.
For $k=1$, the set $U_c \cap \{ x_{j_1} < a_1 \}$ is
itself a set of the kind $U_c$, and therefore is a regular cell, by Lemma~\ref{general_box}.
Then the set $U_c \cap \{ x_{j_1} > a_1 \}$ is a regular cell due to Corollary~\ref{complement}.

By the inductive hypothesis, every non-empty set of the kind
\begin{equation}\label{eq:intersection_box}
U_{c}^{(k-1)}:= U_c \cap \bigcap_{1 \le \nu \le k-1}  \{ x_{j_\nu} \sigma_\nu a_\nu \}
\end{equation}
is a regular cell.
Also by the inductive hypothesis, replacing $c_{j_k}$ by $a_k$ if $a_k < c_{j_k}$, every set
$U_{c}^{(k-1)} \cap \{ x_{j_k} < a_k \}$ is a regular cell.
Since both $U_{c}^{(k-1)}$ and $U_{c}^{(k-1)} \cap \{ x_{j_k} < a_k \}$
are regular cells, so is $U_{c}^{(k-1)} \cap \{ x_{j_k} > a_k \}$,
by the Corollary~\ref{complement}, which completes the induction.
\end{proof}

In the similar statement in the case $n>5$ we need to assume that a generic point $c \in U$.

\begin{lemma}\label{cutting2}
Let $U$ be a semi-monotone set in $\Real_{+}^{n}$, with $n>5$,
let generic $c=(c_1,\dots,c_n) \in U$, and let $U_c:=U \cap \{x_1<c_1,\dots,x_n<c_n\}$.
Then
$$U_c \cap \bigcap_{1 \le \nu \le k}  \{ x_{j_\nu} \sigma_\nu a_\nu \},$$
for a small $\| c \|$,
for any $j_\nu \in \{ 1, \ldots , n \}$, $\sigma_\nu \in \{ <,> \}$, and for any generic
$a_1 \ge \cdots \ge a_k$, is either empty or a regular cell.
\end{lemma}

\begin{proof}
Lemma~\ref{germ} and the argument from the proof of Lemma~\ref{general_box} implies that $U_c$
is a regular cell.
Induction on $k$.
For $k=1$, both sets, $U_c \cap \{ x_{j_1} > a_1 \}$ and $U_c \cap \{ x_{j_1} < a_1 \}$,
are regular cells, due to Lemma~\ref{both}.
Assume by the inductive hypothesis that every non-empty set of the kind $U_{c}^{(k-1)}$
(see (\ref{eq:intersection_box})) is a regular cell.
Then, due to Lemma~\ref{both}, both sets, $U_{c}^{(k-1)} \cap \{ x_{j_k} > a_k \}$ and
$U_{c}^{(k-1)} \cap \{ x_{j_k} < a_k \}$, are regular cells.
\end{proof}

\begin{lemma}\label{box}
Let $U \subset \Real^n$ be a semi-monotone cell, and let $\y=(y_1, \ldots ,y_n) \in \overline U$.
Then for generic points $a=(a_1, \ldots ,a_n)$,
$b=(b_1, \ldots ,b_n) \in \Real_{+}^{n}$, with small $\| a \|$ and $\| b \|$, the intersection
$$U_{a,b}:= U \cap \bigcap_{1 \le j \le n} \{ -a_j < x_j -y_j <b_j \}$$
is a regular cell.
\end{lemma}

\begin{proof}
Induction on $n$ with the base $n=1$ being obvious.

Translate the point $\y$ to the origin.
Let ${\mathbb P}$ be an octant of $\Real^n$.
By Lemma~\ref{general_box} in the case $n \le 5$, or by Lemma~\ref{cutting2} in the case $n>5$,
for a generic point $c=(c_1, \ldots, c_n) \in {\mathbb P} \cap U$, with a small $\| c \|$, the set
$U_c:= U \cap \{|x_1| <|c_1|, \ldots , |x_n| < |c_n| \}$ is either empty or a regular cell.
Choose such a point $c$ in every octant ${\mathbb P}$.

Choose $(-a_i)$ (respectively, $b_i$) as the maximum (respectively, minimum) among the negative
(respectively, positive) $c_i$ over all octants ${\mathbb P}$.
We now prove that, with so chosen $a$ and $b$, the set $U_{a,b}$ is a regular cell.
Induction on $r=0, \ldots ,n-1$.
For the base of the induction, with $r=0$,
if $d$ is a vertex of
$$\bigcap_{1 \le j \le n} \{ -a_j < x_j <b_j \}$$
belonging to one of the $2^n=2^{n-r}$ octants ${\mathbb P}$, then $U_d$
is either empty or a regular cell
(by Lemma~\ref{cutting1} in the case $n \le 5$, or by Lemma~\ref{cutting2} in the case $n>5$).
Partition the family of all sets of the kind $U_d$ into pairs $(U_{d'},U_{d''})$ so that
$d'_1=a_1$, $d''_1=b_1$ and $d'_i=d''_i$ for all $i =2, \ldots, n$.
Whenever the cells $U_{d'}$, $U_{d''}$ are both non-empty, they have the common $(n-1)$-face
$$U \cap \{ x_1=0, |x_2| < d'_2, \ldots , |x_n|< d'_n \}$$
which, by the inductive hypothesis of the induction on $n$, is a regular cell.
Then, according to Corollary~\ref{union}, the union of the common face and
$U_{d'} \cup U_{d''}$ is a regular cell.
Gluing in this way all pairs $(U_{d'},U_{d''})$, we get a family of $2^{n-1}$
either empty or regular cells.
This family is partitioned into pairs of regular cells each of which has the common regular
cell face in the hyperplane $\{ x_2=0 \}$.
On the last step of the induction, for $r=n-1$, we are left with at most two regular cells
having, in the case of the exactly two cells, the common regular cell face in
the hyperplane $\{ x_n=0 \}$.
Gluing these sets along the common face, we get, by Corollary~\ref{union}, the regular cell
$U_{a,b}$.
\end{proof}

\begin{lemma}\label{box_cutting}
Using the notations from Lemma~\ref{box}, the intersection
\begin{equation}\label{eq:cutting}
V_{a,b}:= U_{a,b} \cap \bigcap_{1 \le \nu \le k}  \{ x_{j_\nu} \sigma_\nu d_\nu \},
\end{equation}
for any $j_\nu \in \{ 1, \ldots , n \}$, $\sigma_\nu \in \{ <,> \}$, and for any generic
$d_1 \ge \cdots \ge d_k$, is either empty or a regular cell.
\end{lemma}

\begin{proof}
Analogous to the proof of Lemmas~\ref{cutting1} in the case of $n \le 5$, and to the
proof of Lemma~\ref{cutting2} in the case of $n>5$.
\end{proof}

\begin{proof}[Proof of Theorem~\ref{regularcell}]
For each point $\y \in \overline U$ choose generic points $a,\> b \in \Real^n$ as in Lemma~\ref{box},
so that the set $U_{a,b}$ becomes a regular cell.
We get an open covering of the compact set $\overline U$ by the sets of the kind
$$A_{a,b}=: \bigcap_{1 \le j \le n} \{ -a_j < x_j -y_j <b_j \},$$
choose any finite subcovering ${\mathcal C}$.
For every $j=1, \ldots , n$ consider the finite set $D_j$ of $j$-coordinates $a_j,\> b_j$ for all
sets $A_{a,b}$ in ${\mathcal C}$.
Let
$$\bigcup_{1 \le j \le n} D_j= \{ d_1, \ldots ,d_k \}$$
with $d_1 \ge \cdots \ge d_k$.
Every set $V_{a,b}$, corresponding to $\{ d_1, \ldots ,d_k \}$
(see (\ref{eq:cutting})), is regular, by Lemma~\ref{box_cutting}, and
$U$ is the union of those $V_{a,b}$ and their common faces, for which $A_{a,b} \in {\mathcal C}$.

The rest of the proof is similar to the final part of the proof of Lemma~\ref{box}.
Use induction on $r=1, \ldots ,n$, within the current induction step of the induction on $n$.
The base of the induction is for $r=1$.
Let $D_1= \{ d_{1,1}, \ldots , d_{1,k_1} \}$ with $d_{1,1} \ge \cdots \ge d_{1,k_1}$.
Partition the finite family of all regular cells $V_{a,b}$,
for all $A_{a,b} \in {\mathcal C}$, into $(|D_1|-1)$-tuples
so that the projections of cells in a tuple on the $x_1$-coordinate are exactly the intervals
\begin{equation}\label{eq:intervals}
(d_{1,k_1}, d_{1,k_1-1}), (d_{1,k_1-1}, d_{1,k_1-2}), \ldots , (d_{1,2}, d_{1,1}),
\end{equation}
and any two cells in a tuple having as projections two consecutive intervals in (\ref{eq:intervals})
have the common $(n-1)$-dimensional face in a hyperplane $\{ x_1= {\rm const} \}$.
This face, by the external inductive hypothesis (of the induction on $n$), is a regular cell.
According to Corollary~\ref{union}, the union of any two consecutive cells and their common face
is a regular cell.
Gluing in this way all consecutive pairs in every $(|D_1|-1)$-tuple, we get a smaller
family of regular cells.
This family, on the next induction step $r=2$, is partitioned into $(|D_2|-1)$-tuples of cells
such that in each of these tuples
two consecutive cells have the common regular cell face in a hyperplane $\{ x_2= {\rm const} \}$.
On the last step, $r=n$, of the induction we are left with one $(|D_n|-1)$-tuple of regular cells
such that two consecutive cells have the common regular cell face in a hyperplane
$\{ x_n= {\rm const} \}$.
Gluing all pairs of consecutive cells along their common faces, we get, by Corollary~\ref{union},
the regular cell $U$.
\end{proof}

\section{Semi-algebraic semi-monotone sets over real closed\\ fields}\label{sec:real-closed}

In this section we prove the regularity of semi-monotone sets for
semi-algebraic sets defined over an arbitrary real closed field
$\R$ which is fixed for the rest of the section.
Accordingly, in the definition of semi-monotonicity, ``connectivity'' refers to
``semi-algebraic connectivity'', while
an $n$-dimensional semi-algebraic regular cell $S \subset \R^n$ is such that
there exists a {\em semi-algebraic} homeomorphism
\[
h:\> (\overline{S},S) \to ([-1,1]^n , (-1,1)^n)
\]
(cf. Definition~\ref{def:cell}).

\begin{definition}
Let $S \subset \R^n$ be a semi-algebraic set.
We say that {\em complexity} of $S$ is bounded by a natural number $N$ if there
exists a quantifier-free first-order formula $\Phi$ defining $S$ such that
$N \geq s d$, where $s$ (respectively, $d$)  is the the number (respectively, maximum degree)
of the polynomials appearing in $\Phi$.
By the complexity of a semi-algebraic map we mean the complexity of its graph.
\end{definition}

The idea of the proof is to show that for a fixed $N$ the statement that ``any semi-monotone
set with complexity $N$ is a regular cell'' can be expressed by a first-order formula of the
theory of $\R$ (with integer coefficients), and therefore is true as long as it is true for
$\R= \Real$, due to the
Tarski-Seidenberg transfer principle (\cite{BCR}, Proposition~5.2.3).
(Note that the direct repetition for arbitrary $\R$ of the proof
from Section~\ref{sec:regular} is probably impossible because $\R$ may be non-archimedean.)

\begin{lemma}\label{le:uniform_homeomorphism}
For any pair of semi-algebraic sets $(T_1,T_2)$, with $T_2 \subset T_1 \subset \R^n$,
there exists a natural-valued function $F(N,n)$ with the following property.
Let pairs $(S_1,S_2)$ and $(T_1,T_2)$ of semi-algebraic sets
be semi-algebraically homeomorphic, where the sets $S_2 \subset S_1 \subset \R^n$
have complexities bounded by $N$, and $S_1$ is closed and bounded.
Then there exists a semi-algebraic homeomorphism
$$f:\> (S_1,S_2) \to (T_1,T_2)$$
with complexity bounded by $F(N,n)$.
\end{lemma}

\begin{proof}
It follows from the theorem on triangulations of semi-algebraic sets (\cite{BCR}, Theorem~9.2.1)
that there exists a natural-valued function $H(N,n)$ having the following property.
There exist a finite simplicial complex $K_1$
having at most $H(N,n)$ simplices, a union of its simplices $K_2$, and a semi-algebraic
homeomorphism
\[
h:\> (S_1,S_2) \to (|K_1|,|K_2|),
\]
such that the complexity of $h$ is also bounded by $H(N,n)$.
Since the number of simplicial complexes having at most $H(N,n)$ simplices
is finite, there is a natural-valued function $G(N,n)$, bounding the complexity of any
semi-algebraic homeomorphism
$$g:\> (|K_1|, |K_2|) \to (T_1,T_2).$$
Thus, there exists a semi-algebraic homeomorphism,
$$f := g \circ h:\> (S_1,S_2) \to  (T_1,T_2)$$
with complexity bounded by some natural-valued function $\phi$ of $H(N,n)$ and $G(N,n)$ which
can be explicitly described using bounds on effective quantifier elimination.
Define $F(N,n):= \phi (H(N,n), G(N,n))$.
\end{proof}

One can consider a semi-algebraic subset in $\R^m \times \R^n$ as a {\em semi-algebraic family
of subsets of} $\R^n$ parameterized by points of $\R^m$.
Using again the theorem on triangulations of semi-algebraic sets, it is
easy to check that the family of semi-algebraic subsets of $\R^n$ of complexity bounded
by $N$, which are semi-monotone, is a semi-algebraic family.

\begin{theorem}\label{th:real-closed}
Let $\R$ be a real closed field.
An open bounded semi-algebraic and semi-monotone subset of $\R^n$
is a semi-algebraic regular cell.
\end{theorem}

\begin{proof}
It suffices to prove the theorem for all open bounded semi-algebraic and semi-monotone
sets of complexity bounded by $N$ for each $N > 0$.
Fix $N$.
Since the family of all such sets is a semi-algebraic family,
Lemma~\ref{le:uniform_homeomorphism}, for
$$(T_1,T_2)=([-1,1]^n, (-1,1)^n),$$
implies that the existence of the required homeomorphism is expressible as a
sentence in the language of the first-order theory of the field $\R$ with integer coefficients.
The Tarski-Seidenberg transfer principle now implies that it suffices
to prove the truth of this sentence for any one particular real closed field.
The theorem follows since we have proved the truth of the sentence for $\R = \mathbb{R}$
in Theorem~\ref{regularcell}.
\end{proof}

\section{Regular cells in the sense of van den Dries are not\\ regular}\label{sec:VDD}

In o-minimality theory the following classes of topological cells and continuous functions
are considered, which are also based on the idea of monotonicity.
In \cite{VDD} these cells and functions are called {\em regular}, we will call them
{\em vdD-regular}.

\begin{definition}[cf. Theorem~\ref{th:band}]
An  {\em (open) cylindrical cell} $X \subset \Real^n$ is an open subset defined by induction
as follows.
For $n=0$, $X$ is the point.
Let $X$ be a cylindrical cell in $\Real^{n-1}$, and $f,g:\> X \to \Real$ be two continuous
functions such that $f(\x) < g(\x)$ for all $\x \in X$.
Then $\{ (\x, t) |\> \x \in Y,\ f(\x) <t <g(\x) \}$ is a cylindrical cell in $\Real^{n+1}$.
\end{definition}

\begin{definition}[\cite{VDD}]
A cylindrical cell $X \subset \Real^n$ is {\em vdD-regular} if for each $1 \le i \le n$,
any two points $\x=(x_1, \ldots ,x_n),\> \y=(y_1, \ldots ,y_n) \in X$ and a point
$\z=(z_1, \ldots ,z_n) \in \Real^n$, such that $x_j=y_j=z_j$ for all $j \neq i$, the condition
$x_i < z_i < y_i$ implies $\z \in X$.

Let $X$ be a vdD-regular cell.
A continuous function $f:\> X \to \Real$ is {\em vdD-regular}
if for each $1 \le i \le n$ it is either strictly increasing or strictly decreasing
or is constant along the coordinate $i$.
Herewith, $f$ is strictly increasing along the coordinate $i$ if for any two points
$\x=(x_1, \ldots ,x_n),\> \y=(y_1, \ldots ,y_n) \in X$, such that $x_j=y_j$ for all $j \neq i$,
and $x_i < y_i$, we have $f(\x) < f(\y)$.
Similarly functions strictly decreasing and functions constant along the coordinate $i$ are defined.
\end{definition}

The following example shows that a vdD-regular cell may not be regular, and that a vdD-regular
function defined on a vdD-regular cell may not be supermonotone (or submonotone).

\begin{example}
Consider the 3-simplex
$$X := \{ (x,y,z) \in \Real^3|\> 0<x,\> 0<y,\> 0< z <1,\> x+y < z \},$$
and the continuous function $h:\> X \to \Real$ defined by
$$h(x,y,z)=(x/z)^2+(y/z)^2.$$
Observe that $X$ is vdD-regular cylindrical cell, while $h$ is vdD-regular function.
It is easy to see that for every $t \in (1/2,1)$ the level set
$$\{ (x,y,z) \in X|\> h(x,y,z)=t \}$$
is not connected, while its closure is a cone with the vertex at the origin
and the base consisting of two disjoint arcs of a circle.
Hence the graph of $h$ itself is not a regular cell.
It follows that the vdD-regular cell
$$\{ (x,y,z,t)|\> (x,y,x) \in X, 0 < t < h(x,y,z) \}$$
is not a regular cell.

Note that the set $\{ (x,y,z) \in X|\> h(x,y,z)> 1/2 \}$ consists of two connected
components, and therefore is not semi-monotone.
Therefore the vdD-regular function $h$ is not supermonotone.
\end{example}

\section{Semi-monotone sets and regular Boolean functions}\label{sec:boolean}

Consider a Boolean function $\psi=\psi(\xi_1,\dots,\xi_n)$ in $n$ Boolean variables
$\xi_j\in\{0,1\}$.
For any $j=1,\dots,n$ and $c\in\{0,1\}$, let $\psi_{j,c}$ be the restriction of $\psi$ to
the subspace ${\mathcal C}_{j,c}=\{\xi_j=c\} \subset \{ 0,1 \}^n$.
Operations
$$E_j(\psi)=\psi_{j,0}\lor\psi_{j,1}\quad \text{and}\quad A_j(\psi)=\psi_{j,0}\land\psi_{j,1}$$
assign to $\psi$ two Boolean functions in $n-1$ variables.
Operations $E_j(\psi)$ and $A_j(\psi)$ can be also defined by formulae
$\exists \xi_j\;\psi(\xi_1,\dots,\xi_n)$ and $\forall \xi_j\;\psi(\xi_1,\dots,\xi_n)$,
respectively.

\begin{definition}\label{regular}
A {\sl regular} Boolean function is defined inductively as follows.
Any univariate Boolean function is regular.
A Boolean function $\psi(\xi_1,\xi_2)$ is regular
if the set $\{\psi=1\}$ is neither $\{(1,0),(0,1)\}$ nor $\{(0,0),(1,1)\}$.
Equivalently, $\psi(\xi_1,\xi_2)$ is regular if $E_1(A_2(\psi))=A_2(E_1(\psi))$ or
$A_1(E_2(\psi))=E_2(A_1(\psi))$.
For $n>2$, a Boolean function $\psi$ on $\{0,1\}^n$
is regular if the following two conditions are satisfied:
\begin{itemize}
\item[(R1)] Restriction of $\psi$ to each Boolean square
$$\{ \xi_{j_1}=c_1,\dots,\xi_{j_{n-2}}=c_{n-2}\},\quad 1\le j_1< \dots<j_{n-2}\le n,$$
is regular.
\item[(R2)]
The functions $E_n(\psi)$ and $A_n(\psi)$ are regular.
\end{itemize}
\end{definition}

\begin{lemma}\label{three}
Let $\psi(\xi_1,\xi_2,\xi_3)$ be a Boolean function
such that the functions $\psi_{j,c}$ are regular for all $j=1,2,3$ and $c\in\{0,1\}$.
\begin{enumerate}
\item
If $E_3(\psi)$ is not regular then $\{\psi=1\}$ is one
of the four sets each consisting of two diagonally opposite vertices of $\{ 0,1 \}^3$.
\item
If $A_3(\psi)$ is not regular then $\{\psi=0\}$ is one
of the four sets each consisting of two diagonally opposite vertices of $\{ 0,1 \}^3$.
\end{enumerate}
\end{lemma}

\begin{proof}
Straightforward checking.
\end{proof}

\begin{theorem}\label{commute}
A Boolean function $\psi(\xi_1,\dots,\xi_n)$ is regular if an only if
the result of any sequence
of operations $E_j$ and $A_k$ applied to $\psi$ does not depend on
the order of the operations.
\end{theorem}

\begin{proof}
We prove the statement by induction on $n$.
Case $n\le 2$ follows immediately from the definition of a regular function.
Let $n\ge 3$.

Suppose that for a function $\psi$ the result of any sequence
of operations $E_j$ and $A_k$ applied to $\psi$ does not depend on the order of the operations.
This immediately implies (R1).
Since $E_n(\psi)$ and $A_n(\psi)$ are functions in $n-1$ variables, they are regular by
the inductive hypothesis, i.e., the condition (R2) is also true.
Hence, $\psi$ is regular.

Conversely, let $\psi$ be a regular function.
For any Boolean function $\chi$ and any $j\ne k$ we have
$E_j(E_k(\chi))=E_k(E_j(\chi))$ and $A_j(A_k(\chi))=A_k(A_j(\chi))$.
Condition (R1) implies that for the regular $\psi$ and any $j\ne k$, the equality
$E_j(A_k(\psi))=A_k(E_j(\psi))$ is true.
Hence we have only to show that the functions $E_j(\psi)$ and $A_j(\psi)$
are regular for each $j<n$.
We will only prove that $\varphi := E_j(\psi)$ is regular.
The proof for $A_j(\psi)$ is similar.

For $j<n$, the functions $\varphi_{n,0}:=E_j(\psi_{n,0})$ and
$\varphi_{n,1}:=E_j(\psi_{n,1})$ are regular due to the induction hypothesis.

Since $E_n(\psi)$ is regular and $E_n(\varphi)=E_j(E_n(\psi))$, the function
$E_n(\varphi)$ is regular due to the induction hypothesis.
Since $A_n(\psi)$ is regular and $A_n(\varphi)=E_j(A_n(\psi))$ by the condition (R1),
the function $A_n(\varphi)$ is regular due to the induction hypothesis.
Hence it remains to show that the restriction of $\varphi$ to any Boolean square $B$
in $\{ 0,1 \}^{n-1}$ is regular.
If $B$ has the value of $\xi_n$ fixed, this follows from the regularity of  $\varphi_{n,0}$
and $\varphi_{n,1}$.

Suppose that the values of all variables except $\xi_n$ and $\xi_k$,
for some $j\ne k<n-1$, are fixed on $B$, and the restriction of $\varphi$ to $B$ is not regular.
Then intersection of $\{\psi=1\}$ with the corresponding Boolean 3-cube $C$ in $\{0,1\}^n$
(with the values of all variables except $\xi_n,\,\xi_k$ and $\xi_j$ fixed)
consists of two diagonally opposite vertices due to Lemma~\ref{three}.
Hence the restriction of $E_n(\psi)$ to projection of $C$ along
$\xi_n$ is not regular, which contradicts regularity of $E_n(\psi)$.
\end{proof}

\begin{corollary}\label{cor:commute}
Any regular Boolean function $\psi$ remains regular under any permutation
of the variables, replacing any $\xi_j$ by $1-\xi_j$, replacing $\xi_j$ by a constant
$c \in \{ 0,1 \}$ for any $j \in \{1, \ldots n \}$, and replacing $\psi$ by $1-\psi$.
\end{corollary}

\begin{proof}
Straightforward.
\end{proof}

Consider the cube $[-1,1]^n \subset \Real^n$ as a union
of $2^n$ closed unit cubes with the common vertex at the origin.
Shifting the center of a unit cube by $(1/2, \ldots ,1/2)$ assigns a point in $\{0,1\}^n$
to this unit cube.
In this way, the unit cubes correspond bijectively to the points of $\{0,1\}^n$.

\begin{theorem}\label{cube}
A Boolean function $\psi (\xi_1, \ldots ,\xi_n) \not\equiv 0$ is regular if and only if
the union $C_\psi$ of closed unit cubes corresponding to points
$(\xi_1, \ldots , \xi_n) \in \{ \psi =1 \}$ is a closed PL $n$-ball.
\end{theorem}

\begin{proof}
Induction on $n$, where the base $n=2$ follows from Definition~\ref{regular}.

Suppose a Boolean function $\psi$ is regular.

Let $C_{\psi +}$ (respectively, $C_{\psi -}$) be the union of unit cubes corresponding to vertices
with $\xi_n=1$ (respectively, $\xi_n=0$).
Since, due to Corollary~\ref{cor:commute}, functions $\psi (\xi_1, \ldots ,\xi_{n-1},1)$ and
$\psi (\xi_1, \ldots ,\xi_{n-1},0)$ are regular, both $C_{\psi +}$ and $C_{\psi -}$ are closed PL
$n$-balls, by the inductive hypothesis.
Due to Proposition~\ref{pr:cor3.16}, it is sufficient to prove that the intersection
$$C_{\psi 0}:=C_{\psi +} \cap C_{\psi -} \subset ([-1,1]^n \cap \{ x_n=0 \})$$
is a closed PL $(n-1)$-ball.

If $C_{\psi 0}$ is pure $(n-1)$-dimensional, then $C_{\psi 0}=C_{A_n(\psi)}$.
Since $A_n(\psi)$ is regular, $C_{\psi 0}$ is PL $(n-1)$-ball, by the inductive hypothesis.

We now show that $C_{\psi 0}$ is indeed pure $(n-1)$-dimensional.
Suppose that, on the contrary,  $C_{\psi 0}$ contains a common $m$-face $F$ of a unit cube in
$C_{\psi +}$ and a unit cube in $C_{\psi -}$, with
$0 \le m < n$, and $F$ is not contained in any common face of a larger dimension.

Let $m>0$.
Then for some $1 \le i \le n-1$ and $c \in \{ 0,1 \}$ the set $F \cap \{ x_i=c \}$
is a common $(m-1)$-face of some unit cubes in $C_{\psi +} \cap \{ x_i=c \}$ and
$C_{\psi +} \cap \{ x_i=c \}$, which is not contained in any common face of a larger dimension.
Hence, for the restriction $\psi_{i,c}$ of $\psi$ on $\{ \xi_i=c \}$, the set $C_{\psi_{i,c}}$
is not a PL $(n-1)$-cube, therefore, by the inductive hypothesis, $\psi_{i,c}$ is not regular.
This contradicts Corollary~\ref{cor:commute}.

Now, let $m=0$.
This can only happen when each of $C_{\psi +}$ and $C_{\psi -}$ consists of just one cube,
and this pair of cubes corresponds to diagonally opposite vertices of $[-1,1]^n$.
Then $\psi$ is not regular, which is a contradiction.

Conversely, suppose for a Boolean function $\psi$ the set $C_{\psi}$ is a PL ball.
Then for all $1 \le i \le n$ and $c \in \{ 0,1 \}$ the sets $C_{\psi_{i,c}}$ are also
PL $(n-1)$-balls, hence, by the inductive hypothesis, all functions $\psi_{i,c}$ are regular.
This implies the condition (R1) for $\psi$.

The set $C_{E_n(\psi)}$ is the projection of $C_{\psi}$ along the coordinate $x_n$, and is,
therefore, a PL $(n-1)$-ball.
Hence $E_n(\psi)$ is a regular function by the inductive hypothesis.
The intersection of two PL $n$-balls $C_{\psi +}$ and
$C_{\psi -}$, defined above, is a PL $(n-1)$-ball, and it coincides with $C_{A_n(\psi)}$.
Therefore $A_n(\psi)$ is a regular function by the inductive hypothesis.
It follows that the condition (R2) is also satisfied, and $\psi$ is regular by the definition.
\end{proof}

\begin{corollary}\label{zero}
For a regular Boolean function $\psi$,
\begin{enumerate}
\item
$A_j(\psi) \equiv 0$ if and only if either $\psi_{j,0}\equiv 0$ or $\psi_{j,1}\equiv 0$;
\item
$E_j(\psi) \equiv 1$ if and only if either $\psi_{j,0}\equiv 1$ or $\psi_{j,1}\equiv 1$.
\end{enumerate}
\end{corollary}

\begin{proof}
(1)\
Let $C_{\psi,j +}$ (respectively, $C_{\psi, j -}$) be the union of unit cubes corresponding
to vertices with $\xi_j=1$ (respectively, $\xi_j=0$).
Since $A_j(\psi)=\psi_{j,0}\land\psi_{j,1}\equiv 0$, the set $C_\psi$ cannot contain two unit cubes
corresponding to vertices differing only by $j$th coordinate.
It follows that if both $C_{\psi,j +}$ and $C_{\psi, j -}$ are nonempty, then
$\dim (C_{\psi,j +} \cap C_{\psi, j -}) <n-1$.
This contradicts to the fact that $C_\psi$ is a PL $n$-ball.

The converse statement is trivial.

(2)\ Follows from (1) and the De Morgan's law: $E_j(\psi) \equiv 1- A_j(- \psi)$.
\end{proof}

\begin{definition}\label{def:assigne-boolean}
Let $p \in \Real^n$.
The finite set ${\mathcal Z}$ of octants with vertex at $p$ {\em corresponds} to a Boolean
function $\psi$ if, when translated to $0$, the octants in ${\mathcal Z}$ contain exactly all unit
cubes in $[-1,1]^n$ corresponding to points $(\xi_1, \ldots , \xi_n) \in \{ \psi =1 \}$.

Let $U \subset \Real^n$.
If the set ${\mathcal Z}$ of all
octants with the vertex at $p$ and having non-empty intersections with $U$ corresponds to
a Boolean function $\psi$, then we say that $\psi$ {\em is the function at $p$ for $U$}.
\end{definition}

\begin{lemma}\label{le:constant}
Let $U \subset \Real^n$ be a non-empty semi-monotone set, and $U'$ be its projection along
the coordinate $x_n$.
If $p' \not\in U'$ and the function $\varphi$ at $p'$ for $U'$ is
$\varphi \equiv 1$, then there exists $p_n \in \Real$ such that
the function $\psi$ at $p=(p',p_n)$ for $U$ is either not regular or $\psi \equiv 1$.
\end{lemma}

\begin{proof}
Let $p' \not\in U'$ and $\varphi \equiv 1$.
Suppose that for every $p_n \in \Real$ the function $\psi$ is regular.
Since $\varphi= E_n(\psi)$ for any $p_n \in \Real$, Corollary~\ref{zero} implies that
for every $p_n$ either $\psi_{j,0}\equiv 1$, or $\psi_{j,1}\equiv 1$.
Observe that $\psi_{j,0}\equiv 1$ for all sufficiently large values of $p_n$, while
$\psi_{j,1}\equiv 1$ for all sufficiently small values of $p_n$.
Therefore there exists an intermediate value of $p_n$ for which $\psi \equiv 1$.
\end{proof}

\begin{theorem}\label{octant}
A non-empty open set $U \subset \Real^n$ is semi-monotone if and only if for
every point $p=(p_1, \ldots ,p_n) \in \Real^n \setminus U$
the Boolean function $\psi$ at $p$ for $U$ is a non-constant regular function.
\end{theorem}

\begin{proof}
Suppose that $U$ is semi-monotone and a point $p=(p_1, \ldots ,p_n) \in \Real^n \setminus U$.
Let $\psi$ be the function at $p$ for $U$, and let ${\mathcal Z}$ correspond to $\psi$.

According to Theorem~\ref{cube}, it is sufficient to prove that the union $C$ of all
unit cubes in $[-1,1]^n$ corresponding to octants from ${\mathcal Z}$ is a closed PL $n$-ball
different from the whole $[-1,1]^n$.
We prove this by induction on $n$ with the base case of $n=1$ being trivial.

Let $C=C_+ \cup C_-$ where $C_+$ (respectively, $C_-$) is the union of all unit cubes
corresponding to vertices of $[-1,1]^n$ with $\xi_n=1$ (respectively, $\xi_n=0$).

The projection $U'_+$ (respectively, $U'_-$) of $U \cap X_{n,>,p_n}$ (respectively,
of $U \cap X_{n,<,p_n}$) along the coordinate $x_n$ is semi-monotone due to
Proposition~\ref{le:projection}.
If $(p_1, \ldots ,p_{n-1}) \in U'_\pm$ then the projection of $C_\pm$ along $x_n$
coincides with $[-1,1]^{n-1}$.
Otherwise, by the inductive hypothesis, the projection of $C_\pm$ along $x_n$ is a
closed PL $(n-1)$-ball.
In any case, the set $C_\pm$ itself is a closed PL $n$-ball.

By Proposition~\ref{pr:cor3.16}, it is sufficient to prove that the intersection $C_0:=C_+ \cap C_-$
is a closed PL $(n-1)$-ball.
The same argument as in the proof of Theorem~\ref{cube} shows that $C_0$ is pure
$(n-1)$-dimensional.
We now prove that the set $C_0$ coincides with the union of the unit $(n-1)$-cubes for
$U \cap X_{n,=,p_n}$.
Indeed, if for two octants
$$D_+:=X_{1, \sigma_1, p_1} \cap \cdots \cap X_{n-1, \sigma_{n-1}, p_{n-1}} \cap X_{n, > ,p_n}$$
and
$$D_-:=X_{1, \sigma_1, p_1} \cap \cdots \cap X_{n-1, \sigma_{n-1}, p_{n-1}} \cap X_{n, < ,p_n},$$
where $\sigma_1, \ldots , \sigma_{n-1} \in \{ <,> \}$,
both intersections $D_+ \cap U$ and $D_- \cap U$ are non-empty, then $U \cap X_{n,=,p_n}$
is also non-empty since
$$U \cap X_{1, \sigma_1, p_1} \cap \cdots \cap X_{n-1, \sigma_{n-1}, p_{n-1}}$$
is semi-monotone and therefore connected.
By the inductive hypothesis, $C_0$ is a closed PL $(n-1)$-ball.

It remains to show that $C \neq [-1,1]^n$.
By the inductive hypothesis, $C_0 \neq [-1,1]^{n-1}$.
It follows that
$$U \cap X_{1, \sigma_1, p_1} \cap \cdots \cap X_{n-1, \sigma_{n-1}, p_{n-1}}
\cap X_{n,=,p_n} = \emptyset$$
for some $\sigma_1, \ldots , \sigma_{n-1} \in \{ <,> \}$.
If $C = [-1,1]^n$, then $U \cap D_+ \neq \emptyset$ and $U \cap D_- \neq \emptyset$.
Hence the semi-monotone set
$$U \cap X_{1, \sigma_1, p_1} \cap \cdots \cap X_{n-1, \sigma_{n-1}, p_{n-1}}$$
is not connected which is a contradiction.

Conversely, let for every $p \in\Real^n\setminus U$ the function $\psi$ at $p$ for $U$
be a non-constant regular function.
We continue the proof by induction on $n$, with the base for $n=1$ being trivial.

Let $U'$ be the projection of $U$ along the coordinate $x_n$.
For every point $p' \in \Real^{n-1} \setminus U'$ the function $\varphi$ at $p'$ for $U'$
coincides with $E_n(\psi)$ for the function $\psi$ at some point $p \in \Real^n \setminus U$
for $U$.
Then, by (R2) in Definition~\ref{regular}, $\varphi$ is regular.
The possibility that $\varphi \equiv 1$ contradicts to Lemma~\ref{le:constant}.
Then, by the inductive hypothesis, $U'$ is semi-monotone.
It follows that the intersection
$$V':= U' \cap X_{j_1,\sigma_1,c_1}\cap \cdots \cap X_{j_k,\sigma_k,c_k}$$
is connected for any $0\le k \le n-1$, any $1\le j_1< \dots <j_k \le n-1$, any
$\sigma_1,\dots,\sigma_k$ in $\{<, =, > \}$, and any $c_1, \dots, c_k$.
Suppose that the intersection
$$V:= U \cap X_{j_1,\sigma_1,c_1}\cap \cdots \cap X_{j_k,\sigma_k,c_k}$$
is not connected.
Then, by the Vietoris-Begle theorem, the fibre of the projection along the coordinate $x_n$
over some point $p'=(p_1, \ldots p_{n-1}) \in V'$
is not connected, i.e., for some $p_n,\ x_n,\ y_n \in \Real$ we have
$(p',x_n),\ (p',y_n) \in U$,  $(p',p_n) \in \Real^n \setminus U$, and $x_n < p_n <y_n$.
It follows that the function $\psi$ at $(p',p_{n-1})$ for $U$ is $\psi \equiv 1$, which is
a contradiction.

It remains to consider the case of the intersection $V \cap X_{n, \sigma_n, c_n}$
for $\sigma_n \in \{ < , = , >\}$.

Let $\sigma_n$ be $=$.
We prove that for a point $p \in \Real^{n-1} \setminus (U \cap X_{n, =, c_n})$,
if the intersection $U \cap X_{n, =, c_n} \neq \emptyset$, then the Boolean function
$\varphi$ at $p$ for $U \cap X_{n, =, c_n}$ is a non-constant regular function.
Since the function $\psi$ at $p$ for $U$ is non-constant regular, according to Theorem~\ref{cube},
the corresponding union $C_{\psi}$ of the unit cubes is a PL $n$-ball.
Then for the union $C_{\varphi}$ of unit $(n-1)$-cubes we have
$C_{\varphi}=C_{\psi +} \cap C_{\psi -}$, otherwise the intersection of $U$ with the cylinder
over the corresponding octant in $\Real^{n-1}$ centered at $p$ would be not connected, which
is a contradiction.
It follows that $C_{\varphi}$ is a PL $(n-1)$-ball, thus $\varphi$ is non-constant regular.
By the inductive hypothesis, $U \cap X_{n, =, c_n}$ is a semi-monotone set, in
particular every set $V \cap X_{n, =, c_n}$ is connected.

Suppose that some set of the kind $V \cap X_{n, \sigma_n, c_n}$, where $\sigma_n \in \{ < , > \}$
is not connected.
Since the set $V$ is connected, the set $V \cap X_{n, =, c_n}$
will be not connected which is a contradiction.
\end{proof}

\section{Appendix: topological background}\label{sec:appendix}

\begin{definition}\label{def:gluing}
Let $Z$ be a closed (open) PL $(n-1)$-ball, $X$, $Y$ be closed (respectively, open) PL $n$-balls,
and $$\overline Z= \overline X \cap \overline Y= \partial X \cap \partial Y.$$
We say that $X \cup Y \cup Z$ is obtained by {\em gluing} $X$ {\em and} $Y$ {\em along} $Z$.
\end{definition}

\begin{proposition}[\cite{RS}, Lemma~1.10]\label{pr:le1.10}
Let $X$ and $Y$ be closed PL $n$-balls and $h:\> \partial X \to \partial Y$ a PL homeomorphism.
Then $h$ extends to a PL homeomorphism $h_1:\> X \to Y$.
\end{proposition}

\begin{proposition}[\cite{RS}, Corollary~$3.13_n$]\label{pr:cor3.13}
Let $X$ be a closed PL $n$-ball, $Y$ be a closed $(n+1)$-ball, $\partial Y$ be its boundary
(the PL $n$-sphere), and let $X \subset \partial Y$.
Then $\overline {\partial Y \setminus X}$ is a PL $n$-ball.
\end{proposition}

\begin{proposition}[\cite{RS}, Corollary~3.16]\label{pr:cor3.16}
Let $X$, $Y$, $Z$ be  closed PL balls, as in Definition~\ref{def:gluing},
and $X \cup Y$ be obtained by gluing $X$ and $Y$ along $Z$.
Then $X \cup Y$ is a closed PL $n$-ball.
\end{proposition}

\begin{proposition}[\cite{Shiota}, Lemma~I.3.8]\label{pr:shiota}
Let $X, Y \subset \Real^n$ be compact polyhedra such that $X$ and $X \cup Y$ are
closed PL $n$-balls.
Let $X \cap Y$ be a closed PL $(n-1)$-ball contained in $\partial X$, and let the interior
of $X \cap Y$ be contained in the interior of $X \cup Y$.
Then $Y$ is a closed PL $n$-ball.
\end{proposition}

\begin{definition}\label{def:shellable}
An $n$-dimensional {\em shellable} cell complex is defined by induction as follows.
\begin{enumerate}
\item
Any PL regular $n$-cell $A$ is a shellable complex.
\item
If $W$ an $n$-dimensional shellable complex, $B$ is a PL regular $n$-cell, and
$C$ is a PL regular $(n-1)$-cell in the boundaries of both $W$ and $B$, then
the result of gluing $W$ and $B$ along $C$ is a shellable complex.
\end{enumerate}
\end{definition}

\begin{proposition}\label{pr:shellable}
Any $n$-dimensional shellable cell complex is a PL regular $n$-cell.
\end{proposition}

\begin{proof}
Follows from Proposition~\ref{pr:cor3.16} by the induction in Definition~\ref{def:shellable}.
\end{proof}

\begin{definition}[\cite{RS}, Ch.~4]
A pair of PL manifolds $(Q^m, Q^n)$, in particular balls or spheres, is {\em proper} if
$Q^n \cap \partial Q^m= \partial Q^n$.
A proper pair is {\em locally flat} if each point $\x \in Q^n$ has a neighbourhood in
$(Q^m,Q^n)$ homeomorphic (as a pair) to an open set in
$(\Real_{+}^{m}, \Real_{+}^{n} \times 0)$.
(It is clear that then the pair $(\partial Q^m, \partial Q^n)$ is also locally flat.)
The {\em standard ball pair} is $([-1,1]^m, [-1,1]^n \times 0)$, and
$(\partial [-1,1]^m, \partial [-1,1]^n \times 0)$ is the {\em standard sphere pair}.
A ball or a sphere pair is {\em unknotted} if it is PL homeomorphic to the appropriate standard pair
of the appropriate dimension.
\end{definition}

\begin{proposition}[Sch\"{o}nflies theorem, \cite{RS}, 3.37]\label{pr:schon}
If $n \neq 4$ then any locally flat pair of PL spheres $(S^n,S^{n-1})$ is unknotted.
\end{proposition}

The following statement is apparently well-known in PL-topology.
However, since we don't have a good reference, we present a proof communicated to us by
N.~Mnev.

\begin{proposition}\label{pr:balls}
If $n \neq 4,5$, then any locally flat pair of closed PL balls $(B^n,B^{n-1})$ is unknotted.
\end{proposition}

\begin{proof}
Let $S^{n-1}$ be the sphere boundary of $B^n$, and $S^{n-2}$ be the sphere boundary of $B^{n-1}$.
Let $B_{+}^{n}$ and $B_{-}^{n}$ be the two parts of $B^n$ separated by $B^{n-1}$, and let
$S_{+}^{n-1}$ and $S_{-}^{n-1}$ be the corresponding parts of $S^{n-1}$.
For $n \neq 5$, Proposition~\ref{pr:schon} implies that the pair of PL spheres
$(S^{n-1}, S^{n-2})$ is unknotted, and $S_{\pm}^{n-1}$ are PL $n$-balls.

Let $(C^n,D^{n-1})$ be the cone pair with the base $(S^{n-1}, S^{n-2})$.
Then $W^n:= B^n \cup C^n$ is a PL $n$-sphere, $V^{n-1}:= B^{n-1} \cup D^{n-1}$
is a PL $(n-1)$-sphere, and the pair $(W^n,V^{n-1})$ is locally flat.
For $n \neq 4$, Proposition~\ref{pr:schon}
implies that the pair of PL spheres $(W^n, V^{n-1})$ is unknotted, and the two parts of
$W^n$ separated by $V^{n-1}$ are PL $n$-balls.
But these two parts are the unions $B_{\pm}^{n} \cup E_{\pm}^{n}$, where $E_{+}^{n}$ and
$E_{-}^{n}$ are cones over $S_{+}^{n-1}$ and $S_{-}^{n-1}$ respectively.
Since $S_{\pm}^{n-1}$ are PL $n$-balls, Proposition~\ref{pr:shiota} implies that
the sets $B_{\pm}^{n}$ are also PL $n$-balls.
\end{proof}

\begin{remark}
Proposition~\ref{pr:balls} is also true in the case $n=5$ but available proofs are more complex,
and we don't need this case here.
\end{remark}

\begin{proposition}\label{pr:acyclic}
For $n \le 4$ any acyclic simplicial subcomplex $X$ of the $n$-simplex $\Delta$ has
a vertex with the acyclic link.
\end{proposition}

\begin{proof}
We will consider only the most complex case of $n=4$.

1. If $X$ is one-dimensional, then, being acyclic, it is a tree.
Then $X$ has a leaf, with the acyclic link of a vertex.

2. If $X$ contains a 3-simplex, say $\delta$,
then there is the only vertex, say $v$, in $\Delta \setminus\> \overline \delta$.
If no simplices, apart from $\delta$ and its faces, are in $X$, then
all vertices of $X$ have acyclic links.
Otherwise, $X$ is homotopy equivalent to a suspension of the link of $v$ in $X$, hence
the link is acyclic.
This covers the cases when $X=\Delta$ and when $\dim X=3$.

3. Suppose that $X$ is two-dimensional.
Since any two 2-simplices in $\Delta$ have a common vertex, the one-dimensional part of $X$
consists of trees which cannot have all leaves at some vertices of 2-simplices of $X$
(otherwise $X$ would have a non-trivial 1-cycle).
Hence either such a tree has a leaf with its adjacent vertex as its acyclic link,
or $X$ is pure two-dimensional.

4. Suppose that $X$ is pure two-dimensional.

4a) There are ten 2-simplices in $\Delta$, and they cannot all be in $X$ since
the 2-skeleton of $\Delta$ is not acyclic.
Removing one of them, we get the 2-skeleton $Z$ of an acyclic 3-dimensional complex consisting of
three 3-simplices.
Hence $Z$ has three independent 2-cycles, and
we have to remove at least three 2-simplices to make $Z$ acyclic.
It follows that $X$ has at most six 2-simplices.

4b) If a vertex $v$ of $X$ has one adjacent 2-simplex in $X$ or two adjacent 2-simplices having
a common edge, then $v$ has a link in $X$ which is a tree.
If $v$ has two adjacent 2-simplices in $X$ without a common edge, then $X$, being acyclic,
consists of just these two 2-simplices, and any vertex in $X$, different from $v$, has an
acyclic link.
It follows that if $X$ does not have any vertices with acyclic links, then for each vertex
$v$ of $X$ there should be at least three 2-simplices having $v$ as a common vertex.
If $v$ has exactly three adjacent 2-simplices, the link $L$ of $v$ is connected,
since any disconnected graph with three edges has at least five vertices.
If $L$ is not acyclic, then it is a triangle (the boundary of a 2-simplex).
Since there are five vertices (otherwise $X$ would be a subcomplex of a three-dimensional
simplex), $X$ must have at least five 2-simplices.

It remains to consider the cases of five and six 2-simplices in $X$.

4c) Let $X$ have exactly five 2-simplices.
If each vertex does not have in $X$ an acyclic link (and hence, by 4b), all links are triangles),
then each edge in $X$ is shared by exactly two 2-simplices.
But this is impossible since there are 15 edges to divide into pairs.

4d) Suppose that $X$ has exactly six 2-simplices.

Since the average number of simplices adjacent to the vertices of $X$ is 18/5, there should be
a vertex $v$ of $X$ with exactly three adjacent 2-simplices.
If the link $L$ of $v$ is not acyclic, then, by 4b), it is a triangle.
Let $w$ be the vertex of $X$ different from $v$ and three vertices of $L$.
Then $v$ is not in the link $M$ of $w$, hence $M$ (being a subset of $L$) is either
acyclic or equal to $L$.
In the latter case $X$ is combinatorially equivalent to a triangular bipyramid,
hence is a non-trivial 2-cycle.
This is a contradiction.
\end{proof}

\begin{example}
The following example shows that Proposition~\ref{pr:acyclic} does not hold for $n=5$.

Consider a hexagon with vertices $0, \ldots ,5$.

Attach the boundary of a 3-simplex with vertices $0,1,2,3$ without the simplex $(013)$.
That is a cell contractible to the union of the hexagon edges $(01), (12), (23)$.
Repeat the same construction, replacing $0,1,2,3$ by $2,3,4,5$, and
then repeat again, replacing $2,3,4,5$ by $4,5,0,1$.

None of these three cells have common 2-simplices, and their common edges are all on the hexagon.
Hence, the union $Y$ of these cells is contractible to the hexagon and is homotopy equivalent
to a circle.

Attach the 2-simplex  $(135)$ to $Y$ making the resulting simplicial complex
$X:= Y \cup (135)$ contractible, and therefore acyclic.
Then the links in $X$ of the vertices $0,2,4$ have cycles of length three, while the links
in $X$ of the vertices $1,3,5$ have cycles of length four.
Thus $X$ is an acyclic subcomplex of five-dimensional simplex $\Delta$ having no vertices
with acyclic links.
\end{example}

\begin{proposition}[\cite{VDD}, Ch.~8, (2.14)]\label{pr:fibres}
Let $X \subset \Real^{m+n}$ be a definable set, and let $\pi:\> \Real^{m+n} \to \Real^m$ be the
projection map.
Then there exist an integer $N>0$ and a definable (not necessarily continuous) map
$f:\> X \to \Delta$, where $\Delta$ is an $(N-1)$-simplex, such that for every $\x \in \Real^m$
the restriction $f_{\x}:\> (X \cap \pi^{-1}(\x))  \to \Delta$ of $f$ to $X \cap \pi^{-1}(\x)$
is a definable homeomorphism onto a union of faces of $\Delta$.
\end{proposition}

\begin{corollary}\label{cor:fibres}
Using the notations from Proposition~\ref{pr:fibres},
let all fibres $X \cap \pi^{-1} (\x)$ be definable compact sets.
Then there is a partition
of $\pi (X)$ into a finite number of definable sets $T \subset \Real^m$ such that all
fibres $X \cap \pi^{-1} (\x)$ with $\x \in T$ are definably homeomorphic, moreover
each of these fibres is definably homeomorphic to the same simplicial complex.
\end{corollary}

\begin{proof}
There is a finite number of different unions of faces in $\Delta$.
Since $f$ is definable, the pre-image of any such union under the map $f \circ \pi^{-1}$
is a definable set.
\end{proof}

\end{document}